\documentclass[reqno]{amsart}

\usepackage{amssymb,amsmath,color}
\usepackage{amsthm}
\usepackage{url}
\usepackage{tikz}
\usepackage{hyperref}
\usepackage{ulem}
\usepackage{soul}
\usepackage{enumitem}
\usepackage{multirow}
\usepackage{graphicx}
\usepackage{bbm}
\usepackage{bbding}
\usepackage{color}
\usepackage{nicefrac}
\usepackage{babel}
\usepackage{mathrsfs}
\usepackage{amsbsy}
\usepackage{amstext}
\usepackage{amsthm}
\usepackage{amssymb}
\usepackage{stmaryrd}
\usepackage{xy}
\xyoption{arrow}
\xyoption{matrix}

\definecolor{red}{rgb}{1,0,0}
\definecolor{blue}{rgb}{.2,.2,.8}
\definecolor{brown}{rgb}{.7,.3,0}

\def\gh{\gcd\{h_\lambda(x)\mid\lambda\vdash n\}}
\def\G{G}
\def\e{\gamma}

\def\sp{\mathrm{sp}}
\def\sr{\mathrm{sr}}
\def\den{\mathrm{den}}
\def\num{\mathrm{num}}
\def\val{\mathrm{val}}
\def\rg{\mathrm{\mathfrak q G}}
\def\rd{\mathrm{\mathfrak q d}}

\def\N{\mathbb{N}}
\def\B{\mathcal{B}}
\def\A{\mathrm{A}\hspace{-.015cm}}

\def\o{\mathbf o}
\DeclareMathOperator{\lb}{\log_2}
\newcommand{\mi}{\mathrm{i}}

\newtheorem{theorem}{Theorem}[section]
\newtheorem{corollary}[theorem]{Corollary}
\newtheorem{proposition}[theorem]{Proposition}
\newtheorem{conjecture}{Conjecture}

\newtheorem{lemma}[theorem]{Lemma}
\newtheorem*{theorem*}{Theorem}
\newtheorem*{proposition*}{Proposition}

\theoremstyle{definition}

\newtheorem{example}{Example}
\newtheorem{remark}{Remark}



\begin{document}

\allowdisplaybreaks

\title{Reciprocals of Subsum Polynomials}
\author[C. Ballantine]{Cristina Ballantine}\address{Department of Mathematics and Computer Science\\ College of the Holy Cross \\ Worcester, MA 01610, USA \\} 
 \email{cballant@holycross.edu} 
 \author[G. Beck]{George Beck} \address{Department of Mathematics and Statistics\\ Dalhousie University\\ Halifax, NS, B3H 4R2, Canada \\} \email{george.beck@gmail.com} 
\author[B. Feigon]{Brooke Feigon}\thanks{B.F. partially supported by Simons Foundation Collaboration Grant 635835}\address{Department of Mathematics\\
The City College of New York,
CUNY\\
New York, NY 10031, USA \\}\email{bfeigon@ccny.cuny.edu}
\author[K. Maurischat]{Kathrin Maurischat}\address{Fachgruppe Mathematik\\ RWTH Aachen University\\ Pontdriesch 12-16, 52056 Aachen, Germany\\} \email{kathrin.maurischat@art.rwth-aachen.de }
\begin{abstract} We introduce the subsum polynomial of a partition $\lambda=(\lambda_1, \lambda_2, \ldots, \lambda_k)$  defined by $\sp(\lambda, x)=\prod_{i=1}^k(1+x^{\lambda_i})$. We study the sum of reciprocals of $\sp(\lambda, x)$ over all partitions of $n$. We prove  arithmetic properties of related polynomials and offer connections to other combinatorial objects. 
\end{abstract}
\maketitle
\section{Introduction}

An integer partition $\lambda$ of $n$ is a nonincreasing sequence $\lambda=(\lambda_1, \lambda_2, \ldots, \lambda_k)$ of positive integers that add up to $n$. We write $\lambda\vdash n$ to mean that $\lambda$ is a partition of $n$. The numbers $\lambda_i$ in the sequence are called the parts of $\lambda$. We identify a partition with its multiset of parts. We denote by $P(n)$  the number of partitions of $n$ and set $P(0):=1$.  We occasionally  use the frequency notation of partitions. Exponents of parts represent their multiplicity. For example, $(3^2, 2^4, 1^2)$ is the partition $(3,3,2,2,2,2,1,1)$. For more on the theory of partitions, we refer the reader to \cite{A98}. 

Given a partition  $\lambda=(\lambda_1, \lambda_2, \ldots, \lambda_k)$, we define  the polynomial $$\sp(\lambda, x):=\prod_{i=1}^k(1+x^{\lambda_i}).$$  In analogy with the generating function for the number of partitions with distinct parts, we can view the polynomial $\sp(\lambda, x)$  as the generating function for the number of partitions whose parts form a submultiset of $\lambda$.  This inspired  the name \textit{subsum polynomial} for $\sp(\lambda, x)$. This construction is also reminiscent of Stanton's polynomial $G_\lambda$, the generating function of all partitions that lie inside of $\lambda$ (see \cite{Sta90}).
Clearly, if $\lambda$ is a partition of $n$, then $\deg(\sp(\lambda,x))=n$.

If $m_\lambda(i)$ denotes  the number of times $i$ occurs as a part in $\lambda$, then $$\sp(\lambda, x)=\prod_{i=1}^n(1+x^i)^{m_\lambda(i)}.$$
Numerical experimentation led us to  define and study the rational function $$\sr(n,x):=\sum_{\lambda\vdash n} \frac{1}{\sp(\lambda, x)}.$$ 
For any partition $\lambda\vdash n$, let
\[h_\lambda(x):=\prod_{i\geq 1}(1+x^i)^{\lfloor n/i\rfloor-m_\lambda(i)}\:.\] 
We define
$$\num^*(n,x):=\sum_{\lambda\vdash n}h_\lambda(x) $$ and $$\den^*(n,x):=\prod_{i\geq 1}(1+x^i)^{\lfloor n/i\rfloor}\:.$$ 
Then, $$\sr(n,x)=\frac{\num^*(n,x)}{\den^*(n,x)}.$$

We can also write
\[
\num^*(n,x)=\sum_{\substack{\mu \vdash \sum_{i=1}^n\sigma_1(i)-n\\ m_\mu(i)\leq \lfloor \frac{n}{i}\rfloor}}\sp(\mu,x),
\]
where $\sigma_1(i)$ denotes the sum of divisors of $i$. 

The degree of $\den^*(n,x)$ is $\sum_{i=1}^n i\lfloor n/i\rfloor = \sum_{i=1}^{n}\sigma_1(i)$.  This is sequence A024916 in OEIS \cite{OEIS}. 
We have \begin{align*}\deg(h_\lambda(x))& = \sum_{i=1}^n \sigma_1(i)-\sum_{i=1}^nim_\lambda(i)=\sum_{i=1}^n \sigma_1(i)- n=
\deg(\num^*(n,x)).\end{align*}
  Thus $$\deg(\den^*(n,x))-\deg(\num^*(n,x))= n.$$

Next, we define 
$$\G(n,x):=\gh\:,$$ and let
\[\num(n,x):=\frac{\num^*(n,x)}{G(n,x)}\quad \textrm{ and }\quad \den(n,x):=\frac{\den^*(n,x)}{G(n,x)}\:.\]

\begin{example}
    For $n=3$,  
    \[\begin{array}{c||c|c|c}\lambda& \ \ (3)\ \  & \ \ (2,1) \ \ & \ \ (1^3) \ \ \\ \hline \sp(\lambda)  & 1+x^3 & (1+x)(1+x^2)& (1+x)^3 \\  \hline h_\lambda(x)& (1+x)^3(1+x^2)&(1+x)^2(1+x^3) & (1+x^2)(1+x^3) \end{array}\] 
    
    Then, $G(3,x)=\gcd\{h_\lambda \mid \lambda \vdash 3\}=1+x$ and hence, \[\num(3,x)=\frac{1}{1+x}\left(h_{(3)}(x)+h_{(2,1)}(x)+h_{(1^3)}(x)\right)= 3x^4+2x^3+4x^2+2x+3\:,\]
    and 
    \[\den(3,x)=\frac{(1+x)^3(1+x^2)(1+x^3)}{1+x}=(1+x)^2(1+x^2)(1+x^3)\:.\]
\end{example}

In Section \ref{section:op} we prove results about $\num(n,x)$ and $\den(n,x)$. Below we list some of them. 

A polynomial is  {\it palindromic} if its coefficients  form a palindrome. A sequence $u_1, u_2, \ldots , u_m$ of real numbers is said to be {\it unimodal} if, for some $j$, $u_1 \leq u_2 \leq \ldots \leq u_j \geq u_{j+1} \geq \ldots \geq u_m$. A polynomial is unimodal if its sequence of coefficients is unimodal.

\begin{proposition*}[\ref{prop:ord_p_u}] Let $n$ be a positive integer. Then, 
\begin{enumerate}[label=(\alph*)]
\item    $G(n,x)$, $\num(n,x)$, and $\den(n,x)$ are palindromic. 
 \item  $\den(n,x)$ and $G(n,x)$ are  unimodal.
    \end{enumerate}
\end{proposition*}

\begin{theorem*}[\ref{thm:-1}]
    For all $n\geq 1$, 
        \begin{enumerate} [label=(\alph*)]
        \item $\num(n,-1)=n!$
        \item $\lvert\num(n,\pm\mi)\rvert=2^{\lfloor n/2 
  \rfloor}\lfloor n/2 \rfloor! $
        \item $\lvert\num(n,\zeta_6)\rvert=\begin{cases}
        3^{\lfloor n/3 \rfloor}\lfloor n/3\rfloor !  & n\not\equiv 2 \pmod 3 \\ 4\cdot    3^{\lfloor n/3 \rfloor}\lfloor n/3\rfloor ! & n\equiv 2 \pmod 3,
        \end{cases}$ \\
            where $\zeta_6$ is a primitive sixth root of unity.
    \end{enumerate}
\end{theorem*}

\begin{proposition*}[\ref{conjwas10}]$ $ 
For all $n\geq 1$,
\begin{enumerate}[label=(\alph*)]
    \item   $\den(n,1)=2^{\val_2((2n)!)}$,
    \item  $\num(n,1)=2^{\val_2(n!)}P(n)$.
\end{enumerate}  
\end{proposition*}

Numerical evidence suggests the following conjecture. 
\begin{conjecture}\label{conj:num-irred} For $n\geq 1$, the polynomial $\num(n,x)$ is irreducible over the integers. 
\end{conjecture} 
We also make a somewhat weaker claim for which we have achieved some partial results:
\begin{conjecture}\label{conj:hlambda} For $n\geq 1$,
\[
\gcd\bigl(\num(n,x),\den(n,x)\bigr)=1\:.
\]
\end{conjecture}

\begin{remark}\label{conj12}
 Conjecture \ref{conj:num-irred} implies Conjecture \ref{conj:hlambda}. If $\num(n,x)$ is irreducible, then $\gcd\bigl(\num(n,x),\den(n,x)\bigr)$ must be either $1$ or $\num(n,x)$. However, $\den(n,x)$ is a ratio of $(1+x^i)$'s, thus its irreducible factors are among the  cyclotomic polynomials $\Phi_{2d}(x)$, $d\leq n$ (see Section \ref{sec:prelim}). Hence, if $\gcd\bigl(\num(n,x),\den(n,x)\bigr)=\num(n,x)$, then $\num(n,x)=\Phi_{2d}(x)$ for some $d\leq n$, which  contradicts the degree considerations of Proposition \ref{rd}.
\end{remark}

In Section \ref{sec:binary} we prove some analogous results for binary partitions, that is, partitions whose parts are powers of two. In Section \ref{sec:gen_sums}, we briefly consider a generalization of $\sr(n,x)$, namely $\sum_{\lambda \vdash n} \sp(\lambda, x)^m$ for $m\in \mathbb Z$. In Section \ref{sec:conclud}, we discuss open questions for sums of reciprocals of subsum polynomials defined on restricted sets of partitions.

\section{Preliminaries} \label{sec:prelim}

In this section we establish useful properties of polynomials of the form $1+x^i$. For the convenience of the reader, we give complete proofs.

 We write $\val_2(n)$ for the $2$-valuation of the integer $n$. We denote by $\Phi_d(x)$  the $d$-th cyclotomic polynomial, which is irreducible in $\mathbb{Z}[x]$. 
 \begin{remark} \label{rem:phi}
     Since $x^n-1=\prod_{d|n}\Phi_d(x)$ for any positive integer $n$, and \[ x^{2i}-1=(x^i-1)(1+x^i),\]it follows that 
\begin{equation}\label{eq:cycl}(1+x^i)=\prod_{\substack{d|2i \\ d\nmid i}}\Phi_d(x). \end{equation}  \end{remark}
\begin{lemma}\label{lem:elementary_gcd} $ $

    \begin{enumerate}[label=(\alph*)]
    \item \label{onet} $1+x^i$ is irreducible in $\mathbb{Z}[x]$ if and only if $i=2^k$ for some $k\geq 0$.
    \item \label{gcdb} $\gcd(1+x^n,1+x^m)=\begin{cases}
        1, & \textrm{ if } \val_2(n)\neq \val_2(m)\\
        1+x^{\gcd(n,m)},& \textrm{ if } \val_2(n)= \val_2(m).
    \end{cases}$ \\
     \item \label{divb} In particular, $(1+x^i)\mid (1+x^m)$ if and only if $m=i\cdot j$ for some odd $j$.
    \end{enumerate}    
\end{lemma}
\begin{proof}

\ref{onet} The statement holds because the right hand side of \eqref{eq:cycl} has a single factor if and only if $i=2^k$ for some $k\geq 0$.

 \ref{gcdb} From \eqref{eq:cycl}, it follows that
 \begin{equation}\label{eq:gcd-elementary}\gcd(1+x^n,1+x^m)=\prod_{\substack{d\mid 2\cdot\gcd(n,m)\\ d\nmid n, d\nmid m}}\Phi_d(x).\end{equation}
 Clearly,
 each  $d$  that occurs as an index in the product above satisfies $d\nmid\gcd(n,m)$.
 
 \noindent 
 \underline{Case 1:} $\val_2(n)\neq \val_2(m)$. Assuming $n>m$, let $n=2^j\cdot a\cdot \gcd(n,m)$ and $m=b\cdot \gcd(n,m) $, where $j> 0$ and $a,b$ are odd. Then,  if  $d\mid 2\cdot\gcd(n,m)$ we have $d\mid n$. Hence, the product in \eqref{eq:gcd-elementary} is empty and $\gcd(1+x^n,1+x^m)=1$.
 
  \noindent 
 \underline{Case 2:}  $\val_2(n)=\val_2(m)$. Let $n= a\cdot \gcd(n,m)$ and $m=b\cdot \gcd(n,m) $, where $a,b$ are odd. Then for each $d$ such that $d\mid 2\cdot\gcd(n,m)$ but $d\nmid\gcd(n,m)$ we have $\val_2(d)=\val_2(\gcd(n,m))+1=\val_2(n)+1=\val_2(m)+1$, that is, $d\nmid n$ and $d\nmid m$. We may rewrite \eqref{eq:gcd-elementary} as
 \[\gcd(1+x^n,1+x^m)=\prod_{\substack{d\mid 2\cdot\gcd(n,m)\\ d\nmid \gcd(n,m)}}\Phi_d(x)=1+x^{\gcd(n,m)}.\]

 \ref{divb} Notice that $(1+x^i)\mid(1+x^m)$ is equivalent to $(1+x^i)=\gcd(1+x^i,1+x^m)$, which by part \ref{gcdb} holds if and only if $\val_2(i)=\val_2(m)$ and $i=\gcd(i,m)$, that is, if and only if $m=i\cdot j$ for some odd $j$.
\end{proof}
\begin{remark}
If  $i=\ell\cdot j$ with odd $j$, we have the factorization
\[1+x^i=(1+x^\ell)\cdot\sum_{k=0}^{j-1}(-x^\ell)^{k}=(1+x^\ell)(x^{\ell(j-1)}-x^{\ell(j-2)}+\ldots-x^\ell+1).\]
\end{remark}
\begin{corollary}\label{cor:on_gcd_hlambda} $ $
 The $\gcd$ of any finite collection of products of terms of the form $(1+x^i)$ can be written as a product of terms of the form $(1+x^j)$.
    In particular, $\G(n,x)$ decomposes into the product $$\G(n,x)=\prod_{j=0}^{\lfloor\log_2(n)\rfloor}\gcd\{ \prod_{d \geq 1, d\textrm{ odd}}(1+x^{2^jd})^{\lfloor \frac{n}{2^jd}\rfloor-m_\lambda(2^jd)}\mid \lambda\vdash n\}. $$
  \end{corollary}

We also recall basic facts about palindromic and unimodal polynomials.

\begin{remark}\label{rem:pali-uni} $ $\begin{enumerate}[label=(\alph*)]
\item \label{pal} Products and quotients of palindromic polynomials are palindromic and sums of palindromic polynomials of  the same degree are palindromic.
In particular, for all $i, \ell\geq 1$, the polynomial $(1+x^i)^\ell$ is palindromic. 

\item \label{p_uni} A  product of palindromic,  unimodal polynomials is again  unimodal \cite{A75}. \end{enumerate}
\end{remark}

\section{Ordinary partitions}\label{section:op}

Recall that $$h_\lambda(x)=\prod_{i\geq 1}(1+x^i)^{\lfloor n/i\rfloor-m_\lambda(i)}.$$ 
By Lemma \ref{lem:elementary_gcd} (or also by Corollary \ref{cor:on_gcd_hlambda}) there exists a product decomposition of $\G(n,x)=\gh$ as
\begin{equation}\label{eq:G_decomp}
    \G(n,x)=\prod_{i\geq 1} (1+x^i)^{c_{n,i}}\:.
\end{equation}
Note that this is not a decomposition into coprime factors. Our next goal is to understand the exponents $c_{n,i}$. We first introduce more notation. Given a positive integer $n$,  for each $i=1,\ldots,n$, we define $\e_{n,i}$ as the maximal non-negative integer such that 
\[(1+x^i)^{\e_{n,i}}\mid\G(n,x)\:.\]
Clearly, $c_{n,i}\leq \e_{n,i}$.

The next lemma can be  used to determine the exponents  $c_{n,i}$ recursively for fixed $n$.

\begin{lemma}\label{lem:gcd_hlambda} With  the notation  above, for $n\geq 1$ we have
    \begin{enumerate}[label=(\alph*)]
        \item \label{gammani} For all $i \geq 1$, 
        \[\e_{n,i}=\sum_{j\geq 1 \textrm{ odd}}c_{n,i\cdot j}=\sum_{j>1 \textrm{ odd}}\lfloor\frac{n}{i\cdot j}\rfloor\:.\]
        \item \label{cni}       
        For  $i>\lfloor \frac{n}{3}\rfloor$, $c_ {n,i}=\e_{n,i}=0$.
        \item\label{3i} For $3i>\lfloor \frac{n}{3}\rfloor\geq i$, $c_{n,i}=\e_{n,i}$.        
    \end{enumerate}
\end{lemma}

\begin{proof}[Proof] These are consequences of Lemma \ref{lem:elementary_gcd}:  The term $(1+x^i)$ divides $(1+x^k)$ if and only if $k=i\cdot j$ for some odd $j\geq 1$, and if so, the multiplicity of $(1+x^i)$  in the decomposition of $(1+x^k)$ is one.

\ref{gammani} For any $i$, we have $\e_{n,i}=\sum_{j \geq 1 \textrm{ odd}}c_{n,i\cdot j}$. 
On the other hand, since
\[h_\lambda(x)=\prod_{i\geq 1}(1+x^i)^{\lfloor n/i\rfloor-m_\lambda(i)}\:,\]
 we also have
    \[\e_{n,i}=\min_{\lambda\vdash n}\left(\sum_{j \textrm{ odd}}(\lfloor\frac{n}{i\cdot j}\rfloor-m_\lambda(i\cdot j))\right)\:.\] 
    The sum on the right hand side is minimal for each $\lambda$ such that $m_\lambda(i)=\lfloor\frac{n}{i}\rfloor$ and  for such a $\lambda$ we have $m_\lambda(i\cdot j)=0$ for all $j>1$. The minimality claim holds because, if $\lambda$ has $m_\lambda(i)<\lfloor\frac{n}{i}\rfloor$, then    
      \begin{equation*} \label{ineq:mult}
        \sum_{j>1}m_\lambda(i\cdot j)< \sum_{j>1}jm_\lambda(i\cdot j)\leq \lfloor\frac{n}{i}\rfloor-m_\lambda(i).    \end{equation*}
     Partitions $\lambda$ with $m_\lambda(i)=\lfloor\frac{n}{i}\rfloor$ exist, for  example  $\lambda=(i^{\lfloor\frac{n}{i}\rfloor},1^{n-\lfloor\frac{n}{i}\rfloor})$.
Hence, we obtain     
    \[\e_{n,i}=\sum_{j>1 \textrm{ odd}}\lfloor\frac{n}{i\cdot j}\rfloor.\]

\ref{cni}
 If $i>\lfloor n/3\rfloor$, there is no multiple $k$ of $i$ satisfying $n\geq k=i\cdot j>i$ with $j$ odd. 
 So by \ref{gammani}, $0=\sum_{j>1 \textrm{ odd} }\lfloor\frac{n}{i\cdot j}\rfloor=\e_{n,i}=c_{n,i}$.

\ref{3i} For $3i>\lfloor \frac{n}{3}\rfloor\geq i $, by \ref{gammani} and \ref{cni}, 
$\e_{n,i}=\sum_{j \geq 1 \textrm{ odd}}c_{n,i\cdot j}=c_{n,i}$. 
\end{proof}

For each $i\geq 1$,  $(1+x^i)^{\e_{n,i}}\mid \num^*(n,x)$ and $\e_{n,i}$ is the largest such power as shown below.

\begin{lemma}\label{lemma:on_exponents}
    Let $p(x)=\prod_{i=1}^N(1+x^i)^{a_i}$ and $q(x)=\prod_{i=1}^N(1+x^i)^{b_i}$ be polynomials with exponents $a_i,b_i\geq 0$. Let $\alpha_i$ (resp. $\beta_i$) be the maximal exponent such that $(1+x^i)^{\alpha_i}\mid p(x)$ (resp. $(1+x^i)^{\beta_i}\mid q(x)$). If $\alpha_i=\beta_i$ for all $i$, then $p(x)=q(x)$.
\end{lemma}
\begin{proof}
    Because $\alpha_i=\sum_{j\geq 1 \textrm{ odd}} a_{i\cdot j}$ (resp. $\beta_i=\sum_{j\geq 1 \textrm{ odd}} b_{i\cdot j}$), and because $\alpha_i=a_i$ (resp. $\beta_i=b_i$) for $i> \lfloor N/3\rfloor$, we obtain the exponents $a_i$ (resp. $b_i$) from the $\alpha_i$ (resp. $\beta_i$) by recursion. So if $\alpha_i=\beta_i$ for all $i$, then $a_i=b_i$ for all $i$.
\end{proof}

Using the results of the previous two lemmas, we prove the following about the quotients of  $\G(n,x)$ for consecutive values of $n$. We use the notation \[\rg(n,x) := \frac{\G(n,x)}{\G(n-1,x)}\:.\] For a positive integer $m$, we denote by $\o(m)$ the largest odd divisor of $m$, that is, $$\o(m)=\frac{m}{2^{\val_2(m)}}.$$

\begin{proposition}\label{G:recursion}
 Let $n>1$. Then \[\rg(n,x) = \prod_{\substack{d \mid n\\ \o(n)\nmid d}} (1+x^d)\:,\] In particular, if $n=2^k$ then $\rg(n,x)=1 $. 
\end{proposition}
\begin{proof} 
    Define
    \[\chi(n,k):=\begin{cases}1&\textrm{ if } k\mid n\\0&\textrm{ else, }\end{cases}\:\]
    so that $\lfloor\frac{n}{k}\rfloor=\lfloor\frac{n-1}{k}\rfloor+\chi(n,k)$.
    By Lemma \ref{lem:gcd_hlambda}, for the maximal exponents $\e_{n,i}$ such that $(1+x^i)^{\e_{n,i}}\mid \G(n,x)$ we have
    \[\e_{n,i}=\sum_{j>1 \textrm{ odd}}\lfloor\frac{n}{i\cdot j}\rfloor=\e_{n-1,i}+\sum_{j>1 \textrm{ odd}}\chi(n,i\cdot j)\:.\]
    For each $1\leq i\leq n$, let $\delta_{n,i}$ be the maximal exponent such that $$(1+x^i)^{\delta_{n,i}}\mid \prod_{\substack{d \mid n\\ \o(n)\nmid d}} (1+x^d).$$
     Then  we have 
    \[\delta_{n,i}=\sum_{\substack{j\geq 1 \textrm{ odd }\\ i\cdot j\neq 2^{\val_2(i)-\val_2(n)}n}}\chi(n,i\cdot j)=\sum_{j>1 \textrm{ odd}}\chi(n,i\cdot j).\]
    The last equality holds because both sides have the same number of summands with value $1$ (and all other summands are zero).

    So the maximal exponents $\eta_{ n,i}$ such that $(1+x^i)^{\eta_{n,i}}\mid \prod_{\substack{d \mid n\\ \o(n)\nmid d}} (1+x^d)\cdot\G(n-1,x)$ are
    \[\eta_{n,i}=\delta_{n,i}+\e_{n-1,i}=\e_{n,i}\:.\]
   By Lemma \ref{lemma:on_exponents}, we have $\G(n,x)=\prod_{\substack{d \mid n\\ \o(n)\nmid d}} (1+x^d)\cdot\G(n-1,x)$.
\end{proof}

Next, we examine the exponents $c_{n,i}$ in
$$\G(n,x)=\prod_{i = 1}^{\lfloor \frac{n}{3}\rfloor}(1+x^i)^{c_{n,i}}.$$ 

Let $c_n:=(c_{n,1},c_{n,2},\ldots,c_{n,n})$. By Lemma \ref{lem:gcd_hlambda}, we have  $c_{n,i}=0$ if $i>\lfloor \frac{n}{3}\rfloor$ and $\e_{n,i}=\sum_{j\geq 1 \textrm{ odd}}c_{n,i\cdot j}$. We will now invert this triangular system to find a  closed formula for $c_{n,i}$. First we  need the following identity for $b=2$.

\begin{lemma} \label{lemma:m/2}For all positive integers $m$ and any integral base $b>1$ we have
    \[m=\sum_{{j=1,\, b\nmid j}}^m\lfloor\log_b(\frac{bm}{j})\rfloor.\]
\end{lemma}
\begin{proof}
 We prove this statement  by induction on $m$. When $m=1$, both sides of the equation are $1$. For the induction step  notice that
 $\lfloor\log_b(\frac{b(m+1)}{j})\rfloor=\lfloor\log_b(\frac{bm}{j})\rfloor$ unless $\frac{b(m+1)}{j}=b^k$ for some $k>0$. But there is a unique $j\leq m+1$ with $b\nmid j$ that satisfies this condition. So, by the induction hypothesis,
 \[\sum_{{j=1, \,  b\nmid j}}^{m+1}\lfloor\log_b(\frac{b(m+1)}{j})\rfloor=m+1\:.\]
\end{proof}
\begin{proposition} \label{g-exp} Let $n$ be a positive integer. 
\begin{enumerate}[label=(\alph*)]
\item\label{expo_explicit} For  $1\leq i\leq n$, 
\[c_{n,i}=\lfloor\frac{n}{i}\rfloor-\lfloor\lb(\frac{n}{i})\rfloor-1\:.\]
\item\label{cni1} $c_{n,i}>0$ for $i=1,\ldots,\lfloor n/3 \rfloor$,  and $c_n$ is weakly decreasing.
\end{enumerate}
\end{proposition}

\begin{proof} 
\ref{expo_explicit} We show that $\lfloor\frac{n}{i}\rfloor-\lfloor\lb(\frac{n}{i})\rfloor-1$ satisfies \ref{gammani}--\ref{3i} of Lemma \ref{lem:gcd_hlambda}. That is, $c_{n,i}$ and $\lfloor\frac{n}{i}\rfloor-\lfloor\lb(\frac{n}{i})\rfloor-1$ satisfy the same recursion with the same initial conditions and thus must be equal.  For $i>\lfloor\frac{n}{3}\rfloor$, obviously $\lfloor\frac{n}{i}\rfloor-\lfloor\lb(\frac{n}{i})\rfloor-1=0$. Integers $n,i$ satisfying  $3i>\lfloor n/3\rfloor\geq i$ are related by $n=ai+r$ with $0\leq r<i$ and $a\in\{3,4,\ldots,8\}$. For those
\[c_{n,i}=\e_{n,i}=\lfloor\frac{n}{3i}\rfloor+\lfloor\frac{n}{5i}\rfloor+\lfloor\frac{n}{7i}\rfloor=\lfloor\frac{a}{3}\rfloor+\lfloor\frac{a}{5}\rfloor+\lfloor\frac{a}{7}\rfloor\:,\]
and 
\[\lfloor\frac{n}{i}\rfloor-\lfloor\lb(\frac{n}{i})\rfloor-1=a-\lfloor\lb(a)\rfloor-1\:.\]
We check case-by-case  the claim $c_{n,i}=a-\lfloor\lb(a)\rfloor-1$ holds for $a\in\{3,\ldots,8\}$.
So $\lfloor\frac{n}{i}\rfloor-\lfloor\lb(\frac{n}{i})\rfloor-1$ satisfies \ref{cni} and \ref{3i} of Lemma \ref{lem:gcd_hlambda}.
Next we check that it satisfies the identity  in \ref{gammani} of Lemma \ref{lem:gcd_hlambda}, that is, we show 
\begin{equation*}\sum_{j\geq 1 \textrm{ odd}}^{ \lfloor n/i \rfloor }(\lfloor\frac{n}{i\cdot j}\rfloor-\lfloor\lb(\frac{n}{i\cdot j})\rfloor-1)=\sum_{j>1 \textrm{ odd}}\lfloor\frac{n}{i\cdot j}\rfloor, \ \ \text{for all}\ \  n\geq i\:.\end{equation*} 
The above identity is equivalent to  the following statement.   
\[\lfloor\frac{n}{i}\rfloor=\sum_{j\geq 1\textrm{ odd}}^{\lfloor\frac{n}{i}\rfloor}(\lfloor\lb(\lfloor\frac{n}{i\cdot j}\rfloor)\rfloor+1), \ \ \text{for all}\ \  n\geq i\:,\]
which is true by Lemma \ref{lemma:m/2}.
This finishes the proof for \ref{expo_explicit}.

    \ref{cni1}   The statement follows immediately  from \ref{expo_explicit}. It can also be deduced    from Lemma \ref{lem:gcd_hlambda}.
\end{proof}

Below,  we show part of the infinite matrix $(c_{n,i})_{n,i\geq 1}$, where dots in rows represent zeros. Because $c_{1,i}=c_{2,i}=0$ for all $i\geq 1$, the first two rows are zero.
$$\begin{pmatrix} 
 . & . & . & . & . & . & . & .  \\ 
 . & . & . & . & . & . & . & .  \\ 
 1 & . & . & . & . & . & . & .  \\ 
 1 & . & . & . & . & . & . & .  \\ 
 2 & . & . & . & . & . & . & .  \\ 
 3 & 1 & . & . & . & . & . & .  \\  
 4 & 1 & . & . & . & . & . & . & \ldots \\ 
 4 & 1 & . & . & . & . & . & .  \\ 
 5 & 1 & 1 & . & . & . & . & .  \\ 
 6 & 2 & 1 & . & . & . & . & .  \\ 
 7 & 2 & 1 & . & . & . & . & .  \\ 
 8 & 3 & 1 & 1 & . & . & . & .  \\ 
 9 & 3 & 1 & 1 & . & . & . & .  \\ 
10 & 4 & 1 & 1 & . & . & . & .  \\
   &   &   & \vdots 
\end{pmatrix}$$

\begin{corollary}
The number of $1$'s in $c_n$ is    $\lfloor n/3\rfloor - \lfloor n/5\rfloor$. 

More generally,
\begin{enumerate}[label=(\alph*)]
    \item if $j= 2^t-t-1$ for some $t\in \mathbb N$, then the  number of entries equal to $j$ in $c_n$ is $\lfloor \frac{n}{2^t-1}\rfloor-\lfloor \frac{n}{2^t+1}\rfloor$,
    \item if $j\neq 2^t-t-1$ for any $t\in \mathbb N$, then the  number of entries equal to $j$ in $c_n$ is $\lfloor \frac{n}{j+t_0}\rfloor-\lfloor \frac{n}{j+t_0+1}\rfloor$, where $t_0\in \mathbb N$ is such that $2^{t_0}-t_0-1<j<2^{t_0+1}-(t_0+1)-1$.
\end{enumerate}
\end{corollary}
\begin{proof}  Using Proposition \ref{g-exp}, one can check directly that $c_{n,i}=1$ if and only if $\lfloor n/i\rfloor\in\{3,4\}$, that is,  $n/5<i\leq n/3$. Then the number of $1$'s in $c_n$ is $\lfloor n/3\rfloor - \lfloor n/5\rfloor$. 
The general statement is proved similarly.
\end{proof}

We denote by $\sigma_0(n)$ the number of divisors of $n$. Then, Proposition \ref{g-exp}  leads to the following result on  the number of  factors of the form $(1+x^i)$ in $G(n,x)$, counting multiplicities. 

\begin{corollary}\label{sum_c(n,i)}
     For positive integers $n$, we have $$\log_2(\G(n,1))=\sum_{i=1}^n c_{n,i}=\left(\sum_{i=1}^n \sigma_0(i)\right)-\val_2(n!)-n.$$ 
\end{corollary}

\begin{proof} This follows immediately from Proposition \ref{g-exp}\ref{expo_explicit},
   $$\sum_{i=1}^n\sigma_0(i)=\sum_{i=1}^n\lfloor\frac{n}{i}\rfloor \text{\ \ \ \  and \ \ \ \ } \sum_{i=1}^n\lfloor\lb(\frac{n}{i})\rfloor=\sum_{j=1}^n\val_2(j)=\val_2(n!)\: .$$
 The  identity above on the right  is true because for all $k\geq 0$,  \[\#\{j\leq n\mid \val_2(j)=k\}=\#\{i\leq n\mid \lfloor\lb(n/i)\rfloor=k\}\: .\] 
\end{proof}
\begin{remark} 
  Using  Legendre's formula $\val_2(n!)=\sum_{i\geq 1}\lfloor n/2^i \rfloor$, we also have  $$\log_2(G(n,1))=\sum_{i\geq 1}(\lfloor\frac{n}{i}\rfloor-\lfloor\frac{n}{2^i}\rfloor-1).$$ 
\end{remark}
\begin{proposition}\label{prop:den_G}$ $
\begin{enumerate}[label=(\alph*)]
\item\label{denG:expo1} $\den(n,x)=\prod_{i=1}^{n}(1+x^i)^{a_{n,i}}$ with $a_{n,i}=\lfloor\lb(\frac{n}{i})\rfloor+1$.
\item\label{denG:expo2} The maximal exponents $\alpha_{n,i}$ such that $(1+x^i)^{\alpha_{n,i}}\mid\den(n,x)$ are $\alpha_{n,i}=\lfloor\frac{n}{i}\rfloor$.
\end{enumerate}    
\end{proposition}
\begin{proof}
\ref{denG:expo1} By the definition of $\den(n,x)$ and Proposition \ref{g-exp}, 
\[a_{n,i}=\lfloor \frac{n}{i}\rfloor-c_{n,i}=\lfloor\lb(\frac{n}{i})\rfloor+1\:.\]

\ref{denG:expo2} For $\den^*(n,x)$ the maximal exponents $\beta_{n,i}$  such that $(1+x^i)^{\beta_{n,i}}\mid\den^*(n,x)$ are
\[\beta_{n,i}=\sum_{j\geq 1\textrm{ odd}}\lfloor\frac{n}{i\cdot j}\rfloor=\e_{n,i}+\lfloor\frac{n}{i}\rfloor\:, \]
where $\e_{n,i}$ are the corresponding exponents of $G(n,x)$ (see Lemma \ref{lem:gcd_hlambda}). This determines the maximal exponents $\alpha_{n,i}$  such that $(1+x^i)^{\alpha_{n,i}}\mid \den(n,x)$. They are
\[\alpha_{n,i}=\beta_{n,i}-\e_{n,i}=\lfloor\frac{n}{i}\rfloor\:.\]
\end{proof}

\begin{proposition} \label{prop:ord_p_u} Let $n$ be a positive integer. Then 
\begin{enumerate}[label=(\alph*)]
\item \label{numG:palin}   $G(n,x)$, $\num(n,x)$, and $\den(n,x)$ are palindromic. 
 \item \label{den_G:unim} $\den(n,x)$ and $G(n,x)$ are  unimodal.
    \end{enumerate}
\end{proposition}
\begin{proof}

\ref{numG:palin}
 Using Remark \ref{rem:pali-uni} \ref{pal},  the statement follows from \eqref{eq:G_decomp}, Proposition \ref{prop:den_G} and the definition of $\num(n,x)$.

\ref{den_G:unim}      For all $m$, the product $r(m,x)=\prod_{i=1}^m(1+x^i)$ is palindromic and  unimodal (see e.g., \cite{Dyn1950}, \cite{StZa15}).  
Recall that the exponents $a_{n,i}=\lfloor\lb(\frac{n}{i})\rfloor+1$ of $\den(n,x)$ are weakly decreasing (Proposition \ref{prop:den_G} \ref{denG:expo1}). The same holds  for the exponents $c_{n,i}$ of $G(n,x)$ (Proposition \ref{g-exp} \ref{cni1}).
So $\den(n,x)$, respectively $G(n,x)$, is the product of some $r(m,x)$.  Then the statement follows from Remark \ref{rem:pali-uni}.
\end{proof}

Next, we consider the quotient $$\rd(n,x):=\frac{\den(n,x)}{\den(n-1,x)}.$$

\begin{proposition}\label{eq:qd} Let $n\geq 2$. Then 
\begin{equation*}\rd(n,x)=\prod_{k=0}^{\val_2(n)}(1+x^{2^k\o(n)})=\prod_{d \mid n}{\Phi_{2d}(x)}\:.\end{equation*}
Moreover, $\deg\rd(n,x)=2n-\o(n)$.    
\end{proposition}
\begin{proof} 
We have $$\prod_{d\mid n}(1+x^d)=\frac{\den^\ast(n,x)}{\den^\ast(n-1,x)}=\rg(n,x)\cdot\rd(n,x).$$
Then, 
\[\rd(n,x)=\prod_{k=0}^{\val_2(n)}(1+x^{2^k\o(n)})= \prod_{d \mid n}{\Phi_{2d}(x)}\:,\] where the first equality above follows from Proposition \ref{G:recursion} and the second from 
\eqref{eq:cycl}.

The degree is
\[\deg\rd(n,x)=\sum_{k=0}^{\val_2(n)}2^k\o(n)=2n-\o(n).\]    
\end{proof}

There are various ways to calculate the degree of $\den(n,x)$, which leads to several equivalent formulas that, in turn, yield nice sum identities:

Let $s(n)$ be the sequence defined recursively for $n\geq 1$ by 
\[s(2n) = s(n)+2n\textrm{ \ \ \ and\ \ \ } s(2n-1) = 2n-1.\]
 This is sequence $\A129527$ in \cite{OEIS}. This sequence also appears in the generating function for the number $\mathrm{PL}(n)$ of plane partitions of $n$ (see \cite{BM24}): $$\sum_{n\geq 0} \mathrm{PL}(n)x^n=\prod(1+x^n)^{s(n)}.$$

\begin{corollary}\label{dG:recursion} 
    For $n\geq 2$,  $\deg\rd(n,x)=s(n)$.
\end{corollary}
\begin{proof}
By Proposition \ref{eq:qd}, $\deg\rd(2m+1,x)=2m+1$ and
    \[ \deg\rd(2m,x)=4m-\o(m)=\deg\rd(m,x)+2m\:.\]
\end{proof}
\begin{proposition}\label{rd} $ $ Let $n\geq 1$. Then
\[
\deg\den(n,x)=\sum_{i=1}^ns(i)=n(n+1)-\sum_{i=1}^n\o(i)=\frac{n(n+1)}{2}+\sum_{i=1}^n i\lfloor\lb(\frac{n}{i})\rfloor\: ;\]

$$\deg\num(n,x)=n^2-\sum_{i=1}^n\o(i)= \frac{n(n-1)}{2}+\sum_{i=1}^n i\lfloor\lb(\frac{n}{i})\rfloor \:.$$
    \end{proposition}
\begin{proof} Since \[\deg\den(n,x)=1+ \sum_{i=2}^n\deg\rd(i,x)\:,\] the first two equalities for $\deg\den(n,x)$ follow from Corollary \ref{dG:recursion} and Proposition \ref{eq:qd} respectively. 
For the third equality,  we  use Proposition \ref{prop:den_G} \ref{denG:expo1}. Finally, by definition $\deg\num(n,x)=\deg\den(n,x)-n$.
\end{proof}
\begin{proposition}\label{conjwas10}$ $ For all $n\geq 1$,
\begin{enumerate}[label=(\alph*)]
    \item\label{denG_at1}   $\den(n,1)=2^{\val_2((2n)!)}$,
    \item\label{numG_at1}  $\num(n,1)=2^{\val_2(n!)}P(n)$.
\end{enumerate}  
\end{proposition}
\begin{proof}
 \ref{denG_at1}   For $n=1$ the claim is obvious. For $n>1$ we have the recursion $\den(n,x)=\rd(n,x)\den(n-1,x)$. Then, using induction and the explicit formula  in Proposition \eqref{eq:qd}, we have 
 $$\den(n,1)=\rd(n,1)\den(n-1,1)=2^{\val_2(n)+1}\cdot 2^{\val_2((2n-2)!)}.$$
 Since $\val_2(n)+1=\val_2(2n)+\val_2(2n-1)$, the claim follows.

\ref{numG_at1}  For any $\lambda$, we have $h_\lambda(1)=2^{\sum_{i=1}^n\lfloor n/i\rfloor-n}$. Then, using Corollary \ref{sum_c(n,i)}, we have 
\begin{align*}
    \num(n,1)&=\frac{P(n)\cdot 2^{\sum_{i=1}^n\lfloor n/i\rfloor-n}}{G(n,1)}
    =P(n)\cdot 2^{\val_2(n!)}\:.
\end{align*}
\end{proof}

For the remainder of this section we study $\num(n,x)$, in particular, its irreducibility and values at roots of unity.  We set $\num(0,x):=1$ and  give $\num(n,x)$ explicitly for $1\leq n\leq 5$.  
\begin{example}\label{ex:num}
Let $D(n)=\deg\num(n,x)$.
For small $n$  we list  $\num(n,x)=\sum_{k=0}^{D(n)}a_kx^k$ by giving its sequence  of coefficients.
\begin{center}
\begin{tabular}{c|c|l}
$n$&$D(n)$&$(a_0,\ldots, a_{D(n)})$\\
\hline
$1$&$0$&$(1)$\\
$2$&$2$&$(2,2,2)$\\
$3$&$4$&$(3,2,4,2,3)$\\
$4$&$10$&$(5,8,15,14,24,20,24,14,15,8,5)$\\
$5$&$14$&$(7,9,21,14,37,21,51,24,51,21,37,14,21,9,7)$
\end{tabular}
\end{center}
It can be checked directly that for $1\leq n\leq 5$, the polynomial  $\num(n,x)$ is irreducible over the integers, and in particular $\num(n,\zeta_{2d})\neq 0$ for any  $d$. Here, $\zeta_j$ is a primitive $j$th root of unity. 
\end{example}

 We  write $(a^m, \mu)$ for a partition with largest part  equal to $a$ and repeated $m$ times, and the remaining smaller parts forming a partition $\mu$. Given a partition $\lambda$, we define  
\[
q_\lambda(x):=\frac{h_\lambda(x)}{G(n,x)}\:.
\] Then $\num(n,x)=\sum_{\lambda\vdash n}q_\lambda(x)$. 

To simplify the arguments, we introduce the following notation. For positive integers $n, d$, and remainder $r=n-d\lfloor \frac{n}{d}\rfloor$, we  define \begin{align}
       \notag f(n,d,x)&:=\frac{\prod_{i\neq d}(1+x^i)^{\lfloor n/i\rfloor}}{G(n,x)}\cdot\frac{G(r,x)}{\prod_{i}(1+x^i)^{\lfloor r/i\rfloor}}    
       \\ \label{eq:f_den} & =\frac{\den(n,x)}{(1+x^d)^{\lfloor n/d\rfloor}\den(r,x)}\:.\end{align}

To further examine $\num(\zeta_{2d})$, we prove two useful lemmas.

\begin{lemma}\label{lem:qlambda2}  Fix $n\geq 1, d\geq 1$ and  $\lambda \vdash n$. Let $r=n-d\lfloor \frac{n}{d}\rfloor$. Then
\begin{enumerate}[label=(\alph*)]
    \item\label{q-division:1} $(1+x^d)\nmid q_\lambda(x)$ if and only if $\lambda=(d^{\lfloor n/d\rfloor},\mu)$, where $\mu\vdash r$. Thus \[\num(n,\zeta_{2d})=\sum_{\substack{\lambda=(d^{\lfloor n/d\rfloor},\mu)\\ \textrm{ with } \mu\vdash r}}q_\lambda(\zeta_{2d})\:.\]
    \item\label{q-division:2} $(1+x^d)\nmid q_\lambda(x)$ if and only if $\Phi_{2d}(x)\nmid q_\lambda(x)$.
 \item\label{numf} $\num(n,\zeta_{2d})=f(n,d,\zeta_{2d})\,\num(r,\zeta_{2d})$.
 \item\label{num-nonzero} $\num(n,\zeta_{2d})\neq 0$ if and only if $\num(r,\zeta_{2d})\neq 0$. 
\end{enumerate}
\end{lemma}
\begin{proof}
   \ref{q-division:1} 
    By Lemma \ref{lem:gcd_hlambda}, $(1+x^d)^k\mid q_\lambda(x)$ if and only if
    \begin{align*}
        k&\leq \sum_{j\geq 1, \textrm{ odd}}(\lfloor\frac{n}{d\cdot j}\rfloor-m_\lambda(d\cdot j))-\sum_{j>1, \textrm{ odd}}\lfloor\frac{n}{d\cdot j}\rfloor\\
        &= \lfloor\frac{n}{d}\rfloor-\sum_{j\geq 1, \textrm{ odd}}m_\lambda(d\cdot j).
        \end{align*}
    Since $\lfloor\frac{n}{d}\rfloor=\sum_{j\geq 1, \textrm{ odd}}m_\lambda(d\cdot j)$ if and only if $\lfloor\frac{n}{d}\rfloor=m_\lambda(d)$,
    the result follows. 
    
    \ref{q-division:2} It follows from Remark \ref{rem:phi} and Lemma \ref{lem:elementary_gcd} that \[\Phi_{2d}\mid (1+x^i)\quad\Longleftrightarrow \quad(1+x^d)\mid (1+x^i)\:.\]
    Since $q_\lambda(x)=\frac{h_\lambda(x)}{G(n,x)}$ is a ratio of products of polynomials of the form $(1+x^i)$, the result follows.
    
  \ref{numf}  First note that if $\lambda=(d^{\lfloor n/d\rfloor},\mu)$, then \begin{equation*} q_\lambda(x)=f(n,d,x)\cdot q_\mu(x).\end{equation*}
Hence,
    \begin{align*}
      \sum_{\substack{\lambda=(d^{\lfloor \frac{n}{d}\rfloor},\mu)\\ \textrm{ with } \mu\vdash r}}q_\lambda(x)&=f(n,d,x)\cdot\num(r,x)
    \end{align*} and the result follows from \ref{q-division:1}.
    
 \ref{num-nonzero}  
By \ref{q-division:1},  $f(n,d,\zeta_{2d})\neq 0$ and the result follows from \ref{numf}.
\end{proof}

We deduce a partial result related to  Conjecture \ref{conj:hlambda}: 
\begin{corollary}\label{not-div-gcd} Fix  $n\geq 1$ and $d\geq 1$, and let $r=n-d\lfloor\frac{n}{d}\rfloor$. Then $$\Phi_{2d}(x)\nmid \num(n,x)$$ in each of the following cases:
\begin{enumerate}[label=(\alph*)]
    \item\label{r:explicit} $r\in\{0,1,\ldots,5\}$,
    \item\label{r:small} $\deg\num(r,x)<\varphi(2d)$, where $\varphi$ is Euler's totient function.
\end{enumerate}
In particular, in each of the above cases  $\Phi_{2d}(x)\nmid\gcd(\num(n,x),\den(n,x))$.
 \end{corollary}
\begin{proof} We apply Lemma  \ref{lem:qlambda2} \ref{num-nonzero}.

    \ref{r:explicit}  Example \ref{ex:num} shows that $\num(r,\zeta_{2d})\neq0$ for each such $r$.
   
    \ref{r:small} If $\deg\num(r,x)<\varphi(2d)=\deg\Phi_{2d}(x)$, then $\num(r,\zeta_{2d})\neq 0$.
\end{proof}
Concerning Conjecture \ref{conj:hlambda}, 
$\gcd(\num(n,x), \den(n,x))$
 is at most  a product of $\Phi_{2d}(x)\mid (1+x^i)$, $i\leq n$. Moreover, Corollary \ref{not-div-gcd} excludes some $\Phi_{2d}(x)$. It is natural to search for zeros of  $\num(n,x)$ on the unit circle, that is, to use methods similar to the circle method. Some experimental evidence shows that there are many roots of $\num(n,x)$    on the unit circle. However, we could not find any $\zeta_{2d}$ among them.
This can already be observed for small $n$ from Example \ref{ex:num}: the roots of $\num(n,x)$ are on the unit circle for $n\leq 3$, the roots of $\num(4,x)$ are not on the unit circle, and ten zeros of $\num(5,x)$ are on the unit circle.

In the proof of the next lemma, we make use of the identity \[\prod_{k=0}^{\val_2(m)}(1+x^{2^k\o(m)})= \frac{1-x^{2m}}{1-x^{\o(m)}} \:.\]

\begin{lemma}\label{lemma:zeta} Let $n\geq d\geq 1$. Then 
\[
\frac{f(n,d,\zeta_{2d})}{f(n-d,d,\zeta_{2d})}=d\lfloor \frac{n}{d}\rfloor \cdot 
\frac{2}{\prod_{j=0}^{d-1}(1-\zeta_{2d}^{\o(n-j)})}\:.
\]\end{lemma}

\begin{proof} Using  \eqref{eq:f_den} and Proposition \ref{eq:qd}, we obtain
\begin{align*}
  \frac{f(n,d,x)}{f(n-d,d,x)}& =\frac{1}{(1+x^d)}\frac{\den(n,x)}{\den(n-d,x)}\\
    & = \frac{1}{(1+x^d)}\prod_{j=0}^{d-1}\prod_{k=0}^{\val_2(n-j)}(1+x^{2^k\o(n-j)})\: .
\end{align*}

 Among all  the factors $(1+x^{2^k\o(n-j)})$ in the double product, there is a single factor  with $2^k\o(n-j)$ equal to an odd multiple of $d$, namely  $2^{\val_2(d)}\o(n-r) =d\cdot\o(\lfloor n/d\rfloor)$. Moreover, this factor occurs  only once. We have
\[
\frac{1+x^{d\cdot\o(\lfloor n/d\rfloor)}}{1+x^d}=1-x^d+x^{2d}-\dots +x^{(\o(\lfloor n/d\rfloor)-1)d}.
\]
 Then, $$1-\zeta_{2d}^d+\zeta_{2d}^{2d}-\dots +\zeta_{2d}^{(\o(\lfloor n/d\rfloor)-1)d}=\o(\lfloor n/d\rfloor).$$ We evaluate at $\zeta_{2d}$ the  factor of  $ \frac{f(n,d,x)}{f(n-d,d,x)}$ corresponding to  $j=r=n-d\lfloor n/d\rfloor$ and obtain
\begin{align*}
\frac{1}{(1+x^d)} &  \left. 
\prod_{k=0}^{\val_2(d\lfloor{n}/{d}\rfloor)}(1+x^{2^k\o(d\lfloor{n}/{d}\rfloor)})\right|_{x=\zeta_{2d}} \\ &  =\o(\lfloor n/d\rfloor) \prod_{\substack{k=0\\ k\neq \val_2(d)}}^{\val_2(d\lfloor{n}/{d}\rfloor)}(1+\zeta_{2d}^{2^k\o(d\lfloor{n}/{d}\rfloor)})\\ & 
=\lfloor n/d\rfloor \prod_{k=0}^{\val_2(d)-1}(1+\zeta_{2d}^{2^k\o(d\lfloor{n}/{d}\rfloor)})\\ & = \lfloor n/d\rfloor\frac{1-\zeta_{2d}^{2^{\val_2(d)}\o(d\lfloor n/d\rfloor)}}{1-\zeta_{2d}^{\o(d\lfloor n/d\rfloor)}}=\lfloor n/d\rfloor\frac{2}{1-\zeta_{2d}^{\o(d\lfloor n/d\rfloor)}}\:.
\end{align*}

The product of the remaining factors is \[\prod_{\substack{j=0\\ j\neq r}}^{d-1}\prod_{k=0}^{\val_2(n-j)}(1+\zeta_{2d}^{2^k\o(n-j)})=\prod_{\substack{j=0\\ j\neq r}}^{d-1}\frac{1-\zeta_{2d}^{2(n-j)}}{1-\zeta_{2d}^{\o(n-j)}}\:.  \]
Notice that $$\prod_{j=0, j\neq r}^{d-1}{(1-\zeta_{2d}^{2(n-j)})}=\prod_{\substack{j=0\\ j\neq r}}^{d-1}{(1-\zeta_{d}^{(n-j)})}=\prod_{j=0}^{d-1}(1-\zeta_{d}^{j})=\left.\frac{x^d-1}{x-1}\right|_{x=1}=d.$$ 
This completes the proof. 
\end{proof}
\begin{proposition} \label{prop:numr}
For  $n,d\geq 1$ and $r=n-d\lfloor n/d\rfloor$,
\[\num(n,\zeta_{2d})=\frac{(2d)^{\lfloor n/d\rfloor}\lfloor\frac{n}{d}\rfloor!}{\prod_{j=r+1}^n(1-\zeta_{2d}^{\o(j)})}\cdot \num(r,\zeta_{2d})\:.\]
\end{proposition}
\begin{proof}
Since $f(r,d,x)=1$, it follows that $$f(n,d,x)=\prod_{j=1}^{\lfloor n/d\rfloor}\frac{f(r+jd,d,x)}{f(r+(j-1)d,d,x)}.$$
By Lemma \ref{lem:qlambda2}\ref{numf}, we have $\num(n,\zeta_{2d})=f(n,d,\zeta_{2d})\,\num(r,\zeta_{2d})$.
The claim then follows from Lemma \ref{lemma:zeta}.
\end{proof}
\begin{theorem}\label{thm:-1}
    For all $n\geq 1$, 
        \begin{enumerate} [label=(\alph*)]
        \item \label{-1}$\num(n,-1)=n!$
        \item \label{i}$\lvert\num(n,\zeta_4)\rvert=\lvert\num(n,\pm\mi)\rvert=2^{\lfloor n/2 
  \rfloor}\lfloor n/2 \rfloor! $
        \item \label{zeta6}$\lvert\num(n,\zeta_6)\rvert=\begin{cases}
        3^{\lfloor n/3 \rfloor}\lfloor n/3\rfloor !  & n\not\equiv 2 \pmod 3 \\ 4\cdot    3^{\lfloor n/3 \rfloor}\lfloor n/3\rfloor ! & n\equiv 2 \pmod 3.
        \end{cases}$ 
    \end{enumerate}
\end{theorem}
\begin{proof} We apply Proposition \ref{prop:numr}. Recall $\num(0,x)=\num(1,x)=1$ and $\num(2,x)=2(1+x+x^2)$.

\ref{-1}
For $d=1$ we find $\frac{f(j,1,-1)}{f(j-1,1,-1)}=j$, so $\num(n,-1)=n!$.

\ref{i} For $d=2$ we have 
\[\frac{f(r+2j,2,\zeta_4)}{f(r+2(j-1),2,\zeta_4)}=2j \frac{2}{(1-\zeta_4^{\o(r+2j)})(1-\zeta_4^{\o(r+2j-1)})}\:.\]
Since $|1-\zeta_4^{\textrm{odd}}|=\sqrt 2$,
we obtain $\lvert \num(n,\zeta_4)\rvert=2^{\lfloor n/2\rfloor}\lfloor n/2\rfloor!$ as claimed.

 \ref{zeta6}
 For $d=3$ we have
 \[\frac{f(r+3j,3,\zeta_6)}{f(r+3(j-1),3,\zeta_6)}=3j\frac{2}{(1-\zeta_6^{\o(r+3j)})(1-\zeta_6^{\o(r+3j-1)})(1-\zeta_6^{\o(r+3j-2)})}\:.\]
Next, note that $\o(n)\equiv 3\pmod 6$ if and only if $n\equiv 0\pmod 3$.   This implies $|(1-\zeta_6^{\o(r+3j)})(1-\zeta_6^{\o(r+3j-1)})(1-\zeta_6^{\o(r+3j-2)})|=2$. Thus,
   $\lvert \frac{f(r+3j,3,\zeta_6)}{f(r+3(j-1),3,\zeta_6)}\rvert=3j$, and the claim follows by noticing that  $ \lvert\num(2,\zeta_6)\rvert=4$.
\end{proof}

In general, the expression in Proposition \ref{prop:numr} does not simplify as nicely as in Theorem \ref{thm:-1}.

We conclude this section with the following two conjectures for which we have experimental evidence.

\begin{conjecture} 
 Write $\num(n,x)=\num_0(n,x)+\num_1(n,x)$ as the sum of an even and an odd function. Then  $\num_0(n,x)$ is unimodal. 
\end{conjecture}

\begin{conjecture} 
The sequence of coefficients of $\den(n,x)$ is log-concave except for $n=3,5,6,7$.
\end{conjecture}

\section{Binary partitions}\label{sec:binary}

Let $\mathcal{B}(n)$ denote the set of binary partitions of $n$, that is, partitions whose parts are powers of $2$. We set $B(n)=|\mathcal B(n)|$. 

We define $$\sr_\mathcal{B}(n,x):=\sum_{\lambda \in \mathcal{B}(n)} \frac{1}{\sp(\lambda, x)}.$$
We have $$\sr_\mathcal{B}(n,x)=\frac{\num^*_{\mathcal B}(n,x)}{\den^*_{\mathcal B}(n,x)},$$ where 
\[\num^\ast_{\mathcal B}(n,x)=\prod_{i=0}^{\lfloor\lb(n)\rfloor}(1+x^{2^i})^{\lfloor n/2^i\rfloor-m_\lambda(2^i)}\]and
\[\den_{\B}^\ast(n,x)=\prod_{i=0}^{\lfloor\lb(n)\rfloor}(1+x^{2^i})^{\lfloor n/2^i\rfloor}.\]
If $\lambda$ is a binary partition of $n$,  we let $$h_{\mathcal B,\lambda}(x):=\prod_{i=0}^{\lfloor \lb(n)\rfloor}(1+x^{2^i})^{\lfloor n/2^i\rfloor-m_\lambda(2^i)}.$$ Throughout this section, for ease of notation, we let $h_\lambda(x)=h_{\mathcal B, \lambda}(x).$

Then
\[\num^\ast_{\mathcal B}(n,x)=\sum_{\lambda\in\mathcal B(n)}h_\lambda(x)=:\sum_{k=0}^{d(n)}b_k(n)x^k.\]
We set $\num^\ast_{\B}(0,x):=1$.
 For all $\lambda\in\mathcal B(n)$, 
    \[d(n):=\deg \num^*_{\mathcal B}(n,x) =\deg h_\lambda(x).\]

   \begin{remark} We note that $$G_{\mathcal B}(n,x):=\gcd\{h_{\lambda}(x) \mid \lambda\in \mathcal B(n)\}=1.$$  For $n=1$ the statement is trivial. If $n\geq 2$, $(1+x)\mid h_{(2,1^{n-1})}(x)$ but $(1+x)\nmid h_{(1^n)}(x)$. If $1\leq i\leq \lfloor \log_2 (n)\rfloor$, then $(1+x^{2^i})\mid h_{(1^n)}(x)$ but $(1+x^{2^i})\nmid h_\mu(x)$, where $\mu$ is a partition with $m_\mu(2^i)=\lfloor \frac{n}{2^i}\rfloor$. The result follows from Lemma \ref{lem:elementary_gcd}.
   \end{remark}

   In analogy with the ordinary partitions, for the remainder of the section we use \begin{align*} \num_{\mathcal B}(n,x)& :=\frac{\num^*_{\mathcal B}(n,x)}{G_{\mathcal B}(n,x)}= \num^*_{\mathcal B}(n,x)\:,\\
   \den_{\mathcal B}(n,x)& :=\frac{\den^*_{\mathcal B}(n,x)}{G_{\mathcal B}(n,x)}= \den^*_{\mathcal B}(n,x)\:.\end{align*}
We have the following main conjecture for binary partitions. 
\begin{conjecture}\label{bin} For $n\geq 1$, $$\gcd(\num_{\mathcal B}(n,x), \den_{\mathcal B}(n,x))=1.$$
    \end{conjecture}

We begin by  proving several  useful properties of $\num_{\mathcal B}(n,x)$. Here and throughout, if $\mu \subseteq \lambda$ as multisets, we write $\lambda \setminus \mu$ for the partition obtained by removing the parts of $\mu$ from $\lambda$. Moreover, if $\lambda\in \mathcal B(n)$ has $m_\lambda(1)=0$, we denote by $\lambda/2$ the partition obtained from $\lambda$ by halving each part.

\begin{proposition}\label{lem:binary-elementary}
Let $n\in\N$.
    Then
    \begin{enumerate}[label=(\alph*)]
    \item\label{odd} $\num_{\mathcal B}(2n+1,x)=\num_{\mathcal B}(2n,x)$.
    \item \label{d(n)} For all $n$,
    \[d(n)=\sum_{j\geq 1}2^j\cdot\lfloor\frac{n}{2^j}\rfloor.\]
    \item\label{even}
 $d(n)$ is even and  $2^{\lfloor \lb(n)\rfloor}\lfloor \lb(n)\rfloor\leq d(n)\leq n\lfloor \lb(n)\rfloor$. More precisely, $$d(n)=n\lfloor \lb(n)\rfloor-\sum_{j=\val_2(n)+1}^{\lfloor \lb(n)\rfloor}r_j,$$ where $r_j$ is the remainder of $n$ on division by $2^j$.
\item \label{b_k}  $b_k(n)>0$ for all $k=0,\ldots,d(n)$.
\item \label{numb pal} $\num_{\mathcal B}(n,x)$  and $\den_{\mathcal B}(n,x)$ are palindromic. 
\item \label{small_b} The first coefficients of $\num_{\mathcal B}(n,x)$ are
\[b_0(n)=B(n),\qquad b_1(n)
=\sum_{\lambda\in\mathcal B(n)}(n -m_\lambda(1)),\]
\[b_2(n)=\sum_{\lambda\in\mathcal B(n)}\left(\binom{n -m_\lambda(1)}{2} + {\lfloor n/2\rfloor -m_\lambda(2)} \right).\]
\item \label{B-1}
  $\num_\mathcal{B}(n, -1) =  2^{\val_2(n!)}$ for $n>0$. 
\item \label{prim} For  $n>1$, $(1+x+x^2)\mid\num_{\mathcal B}(n,x)$, that is, the primitive third roots of unity are roots of $\num_{\mathcal B}(n,x)$.
\end{enumerate}
\end{proposition}

\begin{proof}
    \ref{odd} If $\lambda\in \mathcal B(2n+1)$, then $m_\lambda(1)\geq 1$. The transformation $\lambda \to \lambda\setminus (1)$ is a bijection from $\mathcal B(2n+1)$ to $\mathcal B(2n)$.  We verify    $h_{\lambda\setminus (1) }(x)= h_\lambda(x)$, which completes the proof.
  
   \ref{d(n)} If $\lambda\in\mathcal B(n)$, $d(n)=\deg h_\lambda(x)$ and  $$\deg h_\lambda(x)=\sum_{i\geq 0} 2^i\cdot(\lfloor n/2^i\rfloor-m_\lambda(2^i))=\sum_{i\geq 1} 2^i\cdot\lfloor n/2^i\rfloor.$$ 
    
 \ref{even}     Follows immediately from \ref{d(n)}.    
    
   \ref{b_k}  If $k$ is even, $x^k$ occurs in $h_{(1^n)}(x)$. If $k$ is odd, $x^k$ occurs in $h_{(2,1^{n-2})}(x)$.
  
    \ref{numb pal} Follows from Remark \ref{rem:pali-uni}. 
   
   \ref{small_b} Let $\lambda\in \mathcal B(n)$ and denote the coefficient of $x^i$  in $h_\lambda(x)$ by $[x^i]h_\lambda$. Then, \begin{align*}
        [x^0]h_\lambda & = 1,\\ [x^1]h_\lambda & = \binom{\lfloor n/2^0\rfloor -m_\lambda(2^0)}{1}=n -m_\lambda(1),\\ [x^2]h_\lambda & = \binom{\lfloor n/2^0\rfloor -m_\lambda(2^0)}{2} + \binom{\lfloor n/2^1\rfloor -m_\lambda(2^1)}{1}.
    \end{align*}
    Summing over all partitions in $\mathcal B(n)$ completes the proof.

 \ref{B-1} Clearly, $\num_{\mathcal B}(1,-1)=1$. If $n>1$, $\lambda\in \mathcal B(n)$ and  $\lambda \neq (1^n)$, then $h_\lambda(-1)=0$. Hence, for $n>1$,  $$\num_\mathcal B(n,-1)=h_{(1^n)}(-1)= \prod_{i\geq 1}2^{\lfloor n/2^i\rfloor}= 2^{\sum_{i\geq 1}\lfloor n/2^i\rfloor}.$$ Since $\lfloor n/2^i\rfloor$ is the number of positive integers $m\leq n$, such that $2^i\mid m$,  it follows that $\sum_{i\geq 1}\lfloor n/2^i\rfloor=\val_2(n!)$.

\ref{prim} Let $n>1$.
We first make the following observation. If $i>1$ and $\lambda\in\mathcal B(n)$ satisfies  $m_\lambda(2^{i})>0$, let $\lambda^{(i)}$ be the partition obtained from $\lambda$ by removing a part equal to $2^i$ and inserting two parts equal to $2^{i-1}$. 
Then
\begin{align*}
h_\lambda(x)+h_{\lambda^{(i)}}(x)=&\prod_{j\neq i,i-1}(1+x^{2^j})^{\lfloor n/2^j\rfloor-m_\lambda(2^j)}\\
&\times(1+x^{2^i})^{\lfloor n/2^i\rfloor-m_\lambda(2^i)}\cdot(1+x^{2^{i-1}})^{\lfloor n/2^{i-1}\rfloor-m_\lambda(2^{i-1})-2}\\
&\times\left((1+x^{2^{i-1}})^2+(1+x^{2^i})\right).
\end{align*}
The last factor $\left((1+x^{2^i})+(1+x^{2^{i-1}})^2\right)=2(1+x^{2^{i-1}}+x^{2^i})$ is divisible by $(1+x+x^2)$. To see this, note that  $\{2^i \pmod 3, 2^{i-1}\pmod 3\}\equiv\{1,2\}$ and thus  the primitive third roots of unity $\zeta_3,\zeta_3^2$ are roots of $(1+x^{2^{i-1}}+x^{2^i})$. Hence, $(1+x+x^2)\mid (h_\lambda(x)+h_{\lambda^{(i)}}(x))$. To complete the proof we write $\mathcal B(n)$ as a disjoint union of pairs $\lambda, \lambda^{(i)}$. Note that $i$ can vary among pairs. We proceed as follows. 

First, we decompose $\mathcal B(n)$ into the disjoint subsets $\mathcal B^{(i)}(n)$, $i=0,\ldots \lfloor\lb(n)\rfloor,$ 
\begin{align*}
    \mathcal B^{(i)}(n)=&\{\lambda\in\mathcal B(n)\mid\lambda_1=2^i\}\:.
 \end{align*}
Thus, $\lambda\in \mathcal B^{(i)}(n)$ if and only if $m_\lambda(2^i)\geq 1$ and $m_\lambda(2^j)=0$ for $j>i$. We decompose each $\mathcal B^{(i)}(n)$ into the disjoint union of subsets $\mathcal B^{(i,m)}(n)$, $m=1, \ldots \lfloor \frac{n}{2^i}\rfloor$. \begin{align*}
    \mathcal B^{(i,m)}(n)=&\{\lambda\in\mathcal B^{(i)}(n)\mid m_\lambda(2^i)=m\}\:.
 \end{align*}

We create pairs $\lambda, \lambda^{(i)}$ using the following algorithm. 
\begin{enumerate}
    \item[(0)] Set $\mathcal B=\mathcal B(n)$, $i=\lfloor \log_2(n)\rfloor$, $m=\lfloor \frac{n}{2^i}\rfloor$. 
    \item[(1)]
    As long as $\mathcal B^{(i,m)}\cap \mathcal B\neq \emptyset$, choose a $\lambda\in \mathcal B^{(i,m)}\cap \mathcal B$ and put 
 $\mathcal B=\mathcal B\setminus \{\lambda, \lambda^{(i)}\}$. 
 \item[(2)] If $\B=\emptyset$ STOP.
Else put $m=m-1$.
\item[(3)] If $m\geq 1$ go to (1).
Else put $i=i-1$, $m=\lfloor\frac{n}{2^i}\rfloor$, and go to (1).
\end{enumerate}

The algorithm terminates when $i=1$.

The pairs $\lambda, \lambda^{(i)}$ are created starting with the largest first part with the largest multiplicity and  proceeding 
 in decreasing order of multiplicities, respectively, the largest part. At each step, the formed pair is removed from the set.  This ensures that no partition is paired twice. Although the choice of $\lambda$ in step (1) is arbitrary, for  fixed $i$ and $m$, at the end of step (1), all remaining partitions in $\mathcal B^{(i,m)}(n)$ have been paired. For example, when $i=\lfloor\log_2(n)\rfloor$, $m=\lfloor \frac{n}{2^i}\rfloor$, at the end of step (1), all partitions in $\mathcal B^{(i,m)}(n)$ have been paired with partitions in $\mathcal B^{(i,m-2)}(n)$ that have at least two parts equal to $2^{i-1}$. For the same $i$, when  $m=\lfloor \frac{n}{2^i}\rfloor-2$, only partitions with $m_\lambda(2^{i-1})\leq 1$ will be considered. The last pair to be formed by the algorithm is when $i=1$, $m=1$: the partition $(2, 1^{(n-2)})$ is paired with the partition $(1^n)$. Since $B(n)$ is even for $n>1$, no partition is left unpaired. 
\end{proof}
Related to Proposition \ref{lem:binary-elementary}\ref{numb pal} we conjecture the following behavior of the coefficients of $\num_{\B}(n,x)$.
\begin{conjecture}
    For $n\geq 2$, $\num_{\B}(n,x)$ is unimodal. 
\end{conjecture}
\begin{conjecture}\label{conj:binlog}
     For $n\geq 6$, $\num_{\B}(n,x)$ is log-concave.
\end{conjecture}

The reformulation of $b_1(n)$ below follows from Proposition \ref{lem:binary-elementary}  \ref{small_b}. 

\begin{corollary} \label{cor_b1} If $n\geq 1$, then  $b_1(n)=n B(n)-\sum_{k=0}^{n-1}B(k)$. 
    \end{corollary}
    \begin{proof}
        From Proposition \ref{lem:binary-elementary}\ref{small_b}, we have $$b_1(n)= n B(n)-\sum_{\lambda\in\mathcal B(n)}m_\lambda(1).$$ Recall the following standard combinatorial argument:  if $\lambda \in \mathcal B(n)$ with $m_\lambda(1)\geq j$, the $j$th occurrence of  $1$ in  $\lambda$ corresponds to the partition  $\lambda\setminus (1^j)\in\mathcal B(n-j)$. 
        This shows that   
 \begin{equation}\label{sum m1} \sum_{\lambda\in\mathcal B(n)}m_\lambda(1)=\sum_{k=0}^{n-1}B(k).\end{equation}
        \end{proof}

 \begin{remark} The sequence $nB(n)$ is sequence $A304909$ in \cite{OEIS}  and gives  the sum of all parts in all partitions in $\mathcal B(n)$. Thus, $b_1(n)$ is also the sum of all parts greater than $1$ in all partitions in $\mathcal B(n)$.
     \end{remark}

We continue with results about the linear and quadratic coefficients, $b_1(n)$ and $b_2(n)$, of $\num_B(n,x)$. For a sequence $s_n$ we denote by $\Delta s_n$ the sequence of forward differences of $s_n$, that is, $$\Delta s_n=s_{n+1}-s_n.$$

\begin{proposition} \label{delta_b1}
     Let $n\geq 0$. Then $$\Delta b_1(n)=\begin{cases} 0 & \text{ if $n$ even}\\ (n+1)B(\frac{n+1}{2})
   & \text{ if $n$ odd.}
\end{cases}$$
    \end{proposition}
    \begin{proof}
If $n=0$ the statement is clearly true. Let $n\geq 1$.   
By Proposition \ref{lem:binary-elementary}\ref{odd}, if $n=2j$ for some $j\geq 1$,   \begin{align*}b_1(2j+1)-b_1(2j) =0.\end{align*}

    If $n=2j+1$ for some $j\geq 0$, by Corollary \ref{cor_b1}, \begin{align*}b_1(2j+2)-b_1(2j+1) =(2j+2)\left(B(2j+2)- B(2j+1)\right).\end{align*}
    The  transformation $\lambda \to \lambda\setminus(1)$ is a bijection from  $\{\lambda\in \mathcal B(2j+2)\mid m_\lambda(1)\geq 1\}$  to $\mathcal B(2j+1)$. Hence, $b_1(2j+2)-b_1(2j+1)$ equals the number of partitions $\lambda\in \mathcal B(2j+2)$ with $m_\lambda(1)=0$. The transformation $\lambda\to\lambda/2$  is a bijection from the set of partitions $\lambda\in \mathcal B(2j+2)$ with $m_\lambda(1)=0$ to $\mathcal B(j+1)$. This completes the proof. 
    \end{proof}
  \begin{corollary} \label{cor:forward_b_1_2n}
       Let $n\geq 0$.  Then $$\Delta b_1(2n)=2(n+1)B(n+1).$$
      \end{corollary} 
      \begin{proof} 
From Proposition \ref{delta_b1} we have $$b_1(2n+2)-b_1(2n)=b_1(2n+2)-b_1(2n+1)=(2n+2)B(n+1).$$
 \end{proof}
 In general, we have the following expression for $b_1(n)$. 
\begin{corollary} \label{cor:b1_new} For $n\geq 0$,
    $\displaystyle b_1(n)=\sum_{k=0}^{\lfloor n/2\rfloor}2kB(k)$.
\end{corollary}
\begin{proof} Since $b_1(0)=0$ and $b(2n+1)=b(2n)$ for all $n\geq 0$, iterating the statement of Corollary \ref{cor:forward_b_1_2n} gives the result. 
    \end{proof}
\begin{proposition} \label{b2 mod 4} Let $n\geq 2$. Then $b_2(2n)\equiv 2\pmod 4$.
    \end{proposition}

\begin{proof} As  in the proof of Corollary \ref{cor_b1}, we have \begin{align}\sum_{\lambda\in \mathcal B(n)}m_\lambda(2)\label{sum m2} =\sum_{k=1}^{\lfloor n/2\rfloor} B(n-2k)\:.\end{align}
 From Proposition \ref{lem:binary-elementary}\ref{small_b}, we have \begin{align*}b_2(2n)  & = \sum_{\lambda\in\mathcal B(2n)}\left(\binom{2n -m_\lambda(1)}{2} + (n -m_\lambda(2)) \right)\\ & = \sum_{\lambda\in\mathcal B(2n)} \frac{4n^2-4nm_\lambda(1)+(m_\lambda(1))^2}{2}+ \sum_{\lambda\in\mathcal B(2n)} \left(\frac{1}{2}m_\lambda(1)-m_\lambda(2)\right).\end{align*} Since $B(2k)=B(2k+1)$, it follows from \eqref{sum m1} and \eqref{sum m2} that the second sum above equals $0$. Hence \begin{align*}b_2(2n)& =\sum_{\lambda\in\mathcal B(2n)} \frac{4n^2-4nm_\lambda(1)+(m_\lambda(1))^2}{2}\\ & = 2n^2 B(2n)-4n\sum_{k=1}^n B(2n-2k)+\frac{1}{2} \sum_{\lambda\in \mathcal B(2n)}(m_\lambda(1))^2. \end{align*} We can show by induction that  $B(j)$ is even for $j\geq 2$.

Next, we consider the last sum: $$\sum_{\lambda\in \mathcal B(2n)}(m_\lambda(1))^2.$$  If $\lambda \in \mathcal B(2n)$, then $m_\lambda(1)$ is even. 
It is straightforward to see that the partitions $\lambda$ with $m_\lambda(1)=2k$ contribute $4k^2 B(n-k)$ to the sum. Hence, since $B(0)=B(1)=1$ we have  \begin{align*}\frac{1}{2}\sum_{\lambda\in \mathcal B(2n)}(m_\lambda(1))^2 = \sum_{k=1}^{n}2k^2 B(n-k) & \equiv 2n^2+2(n-1)^2 \equiv 2\pmod 4.
     \end{align*}
Thus, $b_2(2n)\equiv 2 \pmod 4$.
\end{proof}

\begin{remark}Using the interpretations above and the fact that $B(2n)=\sum_{k=0}^n B(k)$, we obtain \begin{align*}b_2(2n)& =2n^2 B(2n)-4n\sum_{k=1}^n B(2n-2k)+\sum_{k=1}^{n}2k^2 B(n-k)\\ & = 2n^2\sum_{k=0}^n B(k)- 4n\sum_{k=1}^n\sum_{j=0}^{n-k}B(j)+\sum_{k=0}^{n-1}2(n-k)^2 B(k) \\& = 2n^2\sum_{k=0}^n B(k)- 4n\sum_{k=0}^{n-1}(n-k)B(k)+\sum_{k=0}^{n-1}2(n-k)^2 B(k)  \\ & =\sum_{k=0}^n 2k^2B(k).\end{align*}  Comparing the expression for $b_2(2n)$ with  that for $b_1(2n)$ given in Corollary \ref{cor:b1_new}, we see that $b_2(2n)>b_1(2n)$ for $n\geq 2$, and hence $b_2(n)>b_1(n)$ for $n\geq 3$. 

We can also deduce that  $$\Delta b_2(n)=\begin{cases} 0 & \text{ if $n$ even}\\ \frac{(n+1)^2}{2}B(\frac{n+1}{2})
   & \text{ if $n$ odd}
\end{cases}$$ and  $$\Delta b_2(2n)=2(n+1)^2B(n+1).$$\end{remark}

Next, we prove a recurrence relation for $\num_{\B}(2n,x)$.
\begin{theorem}
For all $n\geq 1$, 
\begin{equation*}
\num_\mathcal B(2n,x)=(1+x)^{2n}\cdot\num_\B(n,x^2)+ f_{2n}(x)\cdot\num_{\B}(2n-2,x),
\end{equation*}
where 
\begin{equation*}
    f_{2n}(x)=\prod_{i=1}^{\val_2(2n)}(1+x^{2^i})\:.
\end{equation*}
\end{theorem}
\begin{proof}
Let 
$$\mathcal B_0(2n):=\{ \lambda\in \mathcal B(2n): m_\lambda(1)=0\} 
$$
and 
$$\mathcal B_1(2n):=\{ \lambda\in \mathcal B(2n): m_\lambda(1)>0\}.$$
 Then  $$\mathcal B(2n)=\mathcal B_0(2n)\bigsqcup \mathcal B_1(2n).$$ 
The transformation $\lambda\to\lambda/2$ is a bijection from $\mathcal B_0(2n)$ to $\mathcal B(n)$.  If $\lambda \in \mathcal B_1(2n)$, then  $m_\lambda(1)\geq 2$. Thus, the transformation $\lambda\to \lambda\setminus(1^2)$ is a bijection from $\mathcal B_1(2n)$ to $\mathcal B(2n-2)$.

We define
 \begin{equation*}\label{eq:recursion_for_numast}
     \mathcal H_1(2n,x)=\sum_{\lambda\in  \B_0(2n)}h_\lambda(x)\:,
 \end{equation*}
and
\begin{equation*}
     \mathcal H_2(2n,x)=\sum_{\lambda\in  \B_1(2n)}h_\lambda(x).
 \end{equation*}
Then 
\begin{equation*}
\num_\mathcal B(2n,x)=\mathcal H_1(2n,x)+ \mathcal H_2(2n,x).
\end{equation*}
If  $\lambda\in \mathcal B_0(2n)$, then $\lfloor \frac{2n}{2^i}\rfloor-m_\lambda(2^i)=\lfloor \frac{n}{2^{i-1}}\rfloor-m_{\lambda/2}(2^{i-1})$ for all $i\geq 1$, and thus
    \[h_\lambda(x)=h_{\lambda/2}(x^2)\cdot(1+x)^{2n}\:.\]
Hence, $\mathcal H_1(2n,x)=(1+x)^{2n}\cdot\num_\B(n,x^2)$.

If $\lambda\in\B_1(2n)$,  then
    \begin{align*}
        h_\lambda(x)=&\prod_{i\geq 0}(1+x^{2^i})^{\lfloor 2n/2^i\rfloor-m_\lambda(2^i)}\\
        =& \prod_{i\geq 1}(1+x^{2^i})^{\lfloor 2n/2^i\rfloor-\lfloor(2n-2)/2^i\rfloor}\cdot h_{\lambda\setminus (1^2)}(x)\\
        =& f_{2n}(x)\cdot h_{\lambda\setminus (1^2)}(x)\:.
    \end{align*}
Hence, $\mathcal H_2(2n,x)=f_{2n}(x)\cdot\num_{\B}(2n-2,x)$.
\end{proof}

\begin{remark} Recall that $d(n)=\deg \num_{\mathcal B}(n,x)=\deg h_\lambda(x)$ for all $\lambda\in \mathcal B(n)$.
Clearly
\[\deg f_{2n}(x)=2^{\val_2(2n)}+\ldots+2=2^{\val_2(2n)+1}-2\:,\]
and 
\[d(2n)=\deg\mathcal H_1(2n,x)=\deg\mathcal H_2(2n,x)\:.\]
\end{remark}
 Next, we deduce recurrence relations  and a formula for $d(n)$.
\begin{proposition} \label{prop:d(n)} The following hold.
    \begin{enumerate}[label=(\alph*)]
    \item\label{dn even} For all even $n\geq 2$,
    \[d(n)-d(n-2)=2^{\val_2(n)+1}-2.\]
    \item\label{dnsum} For all $n\geq 2$,
        \[d(n)=2\left(d(\lfloor\frac{n}{2}\rfloor)+\lfloor\frac{n}{2}\rfloor\right).\]
    \item\label{nbin}  If $n=\sum_i a_i2^i$ is the binary representation of $n$, then 
    \[d(n)=\sum_i i\cdot a_i2^i.\]
    \end{enumerate}
\end{proposition}
\begin{remark}
   The recurrence \ref{dnsum} computes $d(n)$ in $\lb(n)$ steps. The expressions in \ref{dnsum} and \ref{nbin} are  given without proof in \cite[A136013]{OEIS}   (for the sequence $d(n)/2$). 
\end{remark}

\begin{proof}
    \ref{dn even}  For  $n=2m\geq 2$, 
    \begin{align*}
       d(n)=\deg \mathcal H_2(2m,x)
       &=\deg f_{2m}(x)+\deg\num_{\B}(2m-2,x)\\
       &=2^{\val_2(n)+1}-2-d(n-2)\:.
    \end{align*}
    
    \ref{dnsum} First assume $n=2m$ to be even.    
    Then
    \begin{align*}
        d(n)=\deg\mathcal H_1(2m,x)
        &= \deg (1+x)^{2m}+\deg \num_{\B}(m,x^2)\\
        &= 2m +2d(m) =2\left(d(\frac{n}{2})+\frac{n}{2}\right)\:.
    \end{align*}
    For $n$ odd,  $d(n)=d(n-1)$ and $\lfloor\frac{n}{2}\rfloor=\lfloor\frac{n-1}{2}\rfloor$, which completes the proof.
    
    \ref{nbin} This follows by induction using  \ref{dnsum}. 
\end{proof}

Next, we give some evidence for the main conjecture of this section, which we rephrase below.
\begin{conjecture}\label{conj:no_primitive_2tothe_ith_root}
    For all $n\geq 2$  and for  all $2^i\leq n$ we have
    \[(1+x^{2^i})\nmid \num_{\mathcal B}(n,x).\]
\end{conjecture}

\begin{proposition}
    Conjecture \ref{bin} is equivalent to  Conjecture \ref{conj:no_primitive_2tothe_ith_root}.
\end{proposition}
\begin{proof}
    By Lemma \ref{lem:elementary_gcd}, the polynomials $(1+x^{2^i})$ are irreducible for all $i\geq 0$.
Thus, $\den_{\mathcal B}(n,x)=\prod_{i\geq 0}(1+x^{2^i})^{\lfloor n/2^i\rfloor}$ is the decomposition of $\den_{\mathcal B}(n,x)$ into irreducible factors. Hence, any irreducible factor of $\gcd(\den_{\B}(n,x),\num_{\B}(n,x))$
 must be of the form $(1+x^{2^i})$ for $2^i\leq n$.
\end{proof}

\begin{theorem}
Conjecture     \ref{conj:no_primitive_2tothe_ith_root} holds for $n=2^k$ and $n=2^k+1$ for $k\geq 0.$
\end{theorem}
\begin{proof} If $n=2^k$, and $i\leq k$ then $\lfloor n/2^i\rfloor=2^{k-i}$. Let $\lambda\in\B(2^k)$ be the partition with $2^{k-i}$ parts all equal to $2^i$.  Then $(1+x^{2^i})\nmid h_\lambda(x)$.  However, $(1+x^{2^i})\mid h_{\mu}(x)$ for all $\mu\neq\lambda$. Hence, $(1+x^{2^i})\nmid\num_{\B}(2^k,x)$. The statement for  $n=2^k+1$ follows from Proposition \ref{lem:binary-elementary} \ref{odd}.
\end{proof}

 As in Lemma \ref{lem:qlambda2},  $(1+x^{2^i})\mid \num_\B(n,x)$ if and only if $(1+x^{2^i})\mid \num_\B(r,x)$, where $r=n-2^i\lfloor \frac{n}{2^i}\rfloor$. Then,  one can show that for $n\geq 2$, $(1+x^{2^i})\nmid \num_\B(n,x)$ for $i=0, 1$ and $2$.

\section{Sums of powers of subsum polynomials } \label{sec:gen_sums}

Let $\beta(n,x,m)=\sum_{\lambda \vdash n} \sp(\lambda, x)^m$. So far we have studied $\beta(n,x,-1)=\sr(n,x)$.

Let $P_k(n)$ be the number of partitions of $n$ into $k$ sorts of parts; the order does not matter for parts of different sizes, but it does matter for  parts of the same size and  of different sorts. For example, for $n=3$ and $k=2$, the $14$ partitions are $$(3_a),(3_b),(2_a,1_a),(2_a,1_b), (2_b,1_a),(2_b,1_b),$$
$$(1_a,1_a,1_a), (1_b,1_a,1_a),(1_a,1_b,1_a),(1_a,1_a,1_b), (1_a,1_b,1_b),$$
$$(1_b,1_a,1_b),(1_b,1_b,1_a), (1_b,1_b,1_b).$$ (See A246935 in  \cite{OEIS}.) 

\begin{proposition}
For $n \geq 1$, 
\begin{enumerate}[label=(\alph*)]
\item \label{beta} $\beta(n,1,m) = P_{2^m}(n)$,
\item \label{beta2} $\beta(2n,-1,m) = P_{2^m}(n)$ and $\beta(2n-1,-1,m) = 0$.
\end{enumerate}
\end{proposition}

\begin{proof} \ref{beta} First note that $\sp(\lambda,1)=\prod_{i=1}^n 2^{m_\lambda(i)}=2^{\sum m_\lambda(i)}=2^{\ell(\lambda)}$, where $\ell(\lambda)$ is the number of parts of $\lambda$. Hence, $$\beta(n,1,m)=\sum_{\lambda \vdash n}2^{m\ell(\lambda)}.$$ To see that this equals $P_{2^m}(n)$, notice that from each ordinary partition $\lambda$ of $n$ (that is, a partition with parts of a single sort), we can create $2^{m\ell(\lambda)}$ partitions with $2^m$ sorts of parts. There are $2^m$ choices for the sorts for the first part, $2^m$ choices for the sorts for the second part, etc.  

\ref{beta2} If $\lambda$ has an odd part, then $\sp(\lambda, -1)=0$. If $\mathcal P_e(n)$ is the set of partitions of $n$ with all parts even, then  $$ \beta(n,-1, m)=\sum_{\lambda\in \mathcal P_e(n)}\sp(\lambda, -1)^m.$$ Since all partitions of $2n-1$ have at least one odd part, $\beta(2n-1,-1,m) = 0$.

The transformation $\lambda\to \lambda/2$ is a bijection from $\mathcal P_e(2n)$  to the set of all partitions of $n$. If $\lambda\in \mathcal P_e(2n)$, we have $\sp(\lambda, -1)=\sp(\lambda,1)=\sp(\lambda/2, 1).$ Thus $\beta(2n,-1,m)=\beta(n,1,m)$ and the result follows from \ref{beta}.
\end{proof}

\begin{proposition}
    Let $n \geq 1$ and denote by $\beta'(n,x,m)$ the derivative of $\beta(n,x,m)$ with respect to $x$. Then 
$$\val_2(\beta'(n,1,k)) = \val_2(n)+\val_2(k)+k-1.$$
\end{proposition}
\begin{proof} 
We have $$\beta'(n,x,k)=\sum_{\lambda\vdash n} k\cdot \sp(\lambda, x)^{k-1}\sp'(\lambda, x).$$ Using logarithmic differentiation, we have $$\sp'(\lambda, x)=\sp(\lambda, x)\sum_{i=1}^nm_\lambda(i)\frac{ix^{i-1}}{1+x^i}.$$ Evaluating at $x=1$, we obtain
\begin{align*}\beta'(n,1,k)& =\sum_{\lambda\vdash n} k\cdot \sp(\lambda, 1)^{k}\sum_{i=1}^nm_\lambda(i)\frac{i}{2}= \frac{1}{2}nk \sum_{\lambda\vdash n}2^{k\ell(\lambda)}= \frac{1}{2}nk 2^k \sum_{\lambda\vdash n}2^{k(\ell(\lambda)-1)}.\end{align*} Since $\lambda=(n)$ is the only partition of $n$ with $\ell(\lambda)=1$, it follows that $\sum_{\lambda\vdash n}2^{k(\ell(\lambda)-1)}$ is odd. Hence, $\val_2(\beta'(n,1,k)) = \val_2(n)+\val_2(k)+k-1.$
\end{proof}

\section{Open Questions}\label{sec:conclud}

Analogously to $\sr_\mathcal{B}(n,x)$ for the binary partitions of $n$, we can define $sr_{\mathcal A}(n,x)$ for any subset $\mathcal A(n)$ of partitions of $n$ and study the questions posed in this article. We offer the following conjectures.

Let $\mathcal{O}(n)$ be the set of odd partitions of $n$, that is, partitions whose parts are all odd. 
\begin{conjecture}\label{conj:odd}
For $n\ge1$, $\num_\mathcal{O}(n,-1)=\o(n!)$.
\end{conjecture}

Let $\mathcal{T}(n)$ be the set of ternary partitions of $n$, that is, partitions whose parts are powers of $3$. 

\begin{conjecture}\label{conj:s}
Let $s(n) = \num_\mathcal{T}(n,-1)$. Then $s(1)=s(2)=1$ and $s(3n)=s(3n+1)=s(3n+2) = 3^{\val_3((3n)!)}$.
\end{conjecture}

\begin{conjecture}\label{conj:t}
Let $t(n) = \num_\mathcal{T}(n,1)$. Then $t(1)=t(2)=1$ and for $n\geq 1$, 
\[
t(3n)=t(3n+1)=t(3n+2) \mbox{ and}\] 
\[
2^{2n}t(n) = t(3n)-t(3n-2).\]
\end{conjecture}

Finally, we note that  one can replace $x^i$ in $\sp(\lambda,x)$ and $\sr(n,x)$ by $x^{\overline i}$ or $x^{\underline i}$, where $\overline i$ and  $\underline i$, are the rising and falling factorials, respectively.  Conjecturally, these polynomials have interesting properties, and we hope the interested reader will explore them. 

\begin{remark}
    After this paper was submitted, \cite{COZ} was posted on arXiv proving Conjectures \ref{conj:hlambda}, \ref{bin}, \ref{conj:no_primitive_2tothe_ith_root}, \ref{conj:odd}, \ref{conj:s} and \ref{conj:t}. We thank the authors for bringing it to our attention and for pointing out errors in  earlier formulations of Conjectures \ref{conj:binlog} and \ref{conj:t}. 
\end{remark}

\end{document}